\documentclass[11pt]{amsart}

\usepackage{lineno,hyperref}
\usepackage{amsmath, amsthm, amssymb, enumerate}
\usepackage{color}
\usepackage{datetime}
\usepackage{fancyhdr}
\usepackage{amssymb,amsmath,amsthm,amsfonts,graphicx,hyperref,url,dsfont}

\usepackage{listings}
\usepackage{xcolor}
\usepackage{hyperref}
\usepackage[figure]{hypcap}

\usepackage{siunitx}
\usepackage{booktabs}
\usepackage[english]{babel}
\usepackage{blindtext}
\RequirePackage{fix-cm}

\usepackage{fullpage}

\newtheorem{thm}{Theorem}
\newtheorem{lem}[thm]{Lemma}

\newtheorem{cor}{Corollary}[section]

\def\R{\mathbb R}

\def\N{\mathbb N}

\def\L{\mathcal L }

\usepackage[T1]{fontenc}
\usepackage[utf8]{inputenc}
\usepackage{combelow}

\usepackage{newunicodechar}

\usepackage{blindtext}
\usepackage{float}

\newcommand{\vertiii}[1]{{\left\vert\kern-0.25ex\left\vert\kern-0.25ex\left\vert #1
		\right\vert\kern-0.25ex\right\vert\kern-0.25ex\right\vert}}

\newcommand{\norm}[2]{{\left\| #1 \right\|}_{#2}}

\newcommand{\ltwo}[1]{\norm{#1}{L^2}}

\newcommand{\linf}[1]{\norm{#1}{L^\infty}}

\begin{document}
\title{Wick-type stochastic parabolic equations with  random potentials}

\author[S. Gordi\'c, T. Levajkovi\'c, Lj. Oparnica]{Sne\v zana Gordi\'c$^1$, Tijana Levajkovi\'c$^2$, \\ Ljubica Oparnica$^3$}
\date{March 30, 2022}

\thanks{$^1$Sne\v zana GORDI\' C, {Faculty of Education, University of Novi Sad, Podgori\v cka 4, 25000 Sombor, Serbia, \texttt{snezana.gordic@pef.uns.ac.rs}}\\
	\indent$^2$Tijana LEVAJKOVI\'C, {Applied Statistics, Institute of Stochastics and Business Mathematics, TU Vienna,  \linebreak Austria, \texttt{tijana.levajkovic@tuwien.ac.at}}\\
\indent$^3$Ljubica OPARNICA, Department of Mathematics: Analysis, Logic and Discrete Mathematics, University of Ghent, Krijgslaan 281 (Building S8), B 9000 Ghent, Belgium,  \texttt{oparnica.ljubica@ugent.be} \& Faculty of Education, University of Novi Sad, Podgori\v cka 4, 25000 Sombor, Serbia
}

\begin{abstract}
The  stochastic parabolic equations with  random  potentials,  driving forces and  initial  conditions are considered. The Wick product is used to give sense to the product of two generalized stochastic processes, and  the existence and uniqueness of  solutions  are proved via the chaos expansion method from  white noise analysis. The estimates on coefficients in the chaos expansion form of the solutions are provided.
\end{abstract}

\maketitle

\section{Introduction and preliminaries} \label{sec1} 
\label{sec1} 

For given generalized stochastic processes of Kondratiev-type $F$ and $G$, and a bounded in space  generalized stochastic process of Kondratiev-type $Q$,  we consider Cauchy problems for  stochastic parabolic evolution equations
\begin{equation}
\label{general problem}
(\partial_t -\mathcal{L})U + Q \lozenge U =F, \quad U |_{t=0} = G,
\end{equation} 
where $\mathcal{L}$ is an elliptic operator acting on the space variable. Since the  unknown generalized stochastic process $U$ is involved in product with another generalized stochastic proces, potential  $Q$, one must give sense to such product. Here, the Wick product denoted by  $\lozenge$ is used, see \cite{HOUZ1996}. 

The special case, when $\mathcal{L}$ is the Laplacian, is the stochastic heat equation with random potential which, due to its various applications in biology, financial mathematics, aerodynamics, structural acoustics, has been widely studied, e.g. \cite{AHR1997,ANV2000,BDP1997,KL2017}. 
Stochastic evolution equations with multiplicative noise are studied in \cite{LPSZ2015} and stochastic evolution problems with polynomial nonlinearities are studied in \cite{LPSZ2018}. 

In this work we study problem \eqref{general problem}. Using the chaos expansion method from the white noise analysis developed in \cite{HOUZ1996} we prove the theorem 
on existence of unique generalized stochastic process and provide estimates of coefficients in chaos expansion form of the solution.

%

To start with, we recall some basic notions from the white noise analysis, and for more details and proofs we refer to \cite{HOUZ1996}. 
Denote by  $\mathcal{I} :=\N^m_0 $ the set of multi-indices having finite number of  nonzero components,  
the zero vector by $\textbf{0}$,
%
and the length of a multi-index $\alpha = (\alpha_1,\alpha_2, \ldots,\alpha_m,0,0,\ldots)$, $\alpha_i \in \N_0$, $i=1, 2,\ldots, m$, $m\in\N$,  by  $|\alpha|= \sum_{i=1}^{m}\alpha_i$.
If $\alpha=(\alpha_1,\alpha_2,\dots)\in \mathcal I$ and $\beta=(\beta_1,\beta_2,\dots) \in \mathcal I$, then $\alpha \leq \beta$ if and only if $\alpha_k \leq \beta_k$ for all $k \in \mathbb{N}$. 
The following lemma collects results needed later. For proofs and details see \cite{HOUZ1996,LPSZ2018,LPSZ2021sub}.
\begin{lem} \label{pomocna lema - ocene}
	Let $\alpha\in \mathcal{I}$ and $k\in \mathbb{N}_0$.  
	
	(a) Let $k\leq |\alpha|$, and denote by $N(\alpha,k)$ the number of possibilities in which a multi-index $\alpha$ can be written as a sum of $k$ strictly smaller and nonzero multi-indices. Then $N(\alpha,k)\leq 2^{k|\alpha|}$.
	
	%
	(b) Define $(2\N)^\alpha: = \prod_{i=1}^{\infty} (2i)^{\alpha_i}$. Then $|\alpha| \leq (2\mathbb N)^\alpha$ and 	
	\begin{equation}\label{2N konv p>1} 
	\sum_{\alpha\in \mathcal I} (2\N)^{-p\alpha} < \infty\quad  \Leftrightarrow \quad  p>1.
	\end{equation} 
	Moreover, for every  $c>0$ there exists $s\geq 0$ such that $c^{|\alpha|}\leq (2\mathbb N)^{s\alpha}$ and   $c^\alpha\leq (2\mathbb N)^{s\alpha}.$
\end{lem}

For $\gamma\in \mathcal{I}$, the $\gamma$th {Fourier-Hermite polynomial} is defined by
$H_{\gamma}(\omega):=\prod_{k=1}^\infty h_{\gamma_k}(\langle\omega,
\xi_k\rangle)$, where $\xi_k$, $k\in \mathbb{N}$, is the Hermite function of order $k$,  and $h_k$, $k \in \mathbb N_0$, is the Hermite polynomial. 
For a normed space $X$,  the tensor product 
$X\otimes (S)_{-p}$ is the space of $X$-valued generalized stochastic processes of Kondratiev-type, and the space $X\otimes (S)_{-1}$ is the inductive limit of  spaces $X\otimes (S)_{-p }$, $p\geq 0$. 
Every $X$-valued generalized stochastic process of Kondratiev-type, $F \in X\otimes (S)_{-1}$, can be  represented in {\em the chaos expansion form} $F(x,\omega) = \sum_{ \gamma\in \mathcal I}f_\gamma(x) H_\gamma(\omega)$,  $f_\gamma\in X$,  with  
$\|F\|^2_{X\otimes (S)_{-p }}: = \sum_{ \gamma\in \mathcal{I}}\|f_\gamma\|^2_X (2\N)^{-p\gamma}$  finite for some $p\geq 0$. If it is finite for $p_0$, then it is finite for all $p\geq p_0$. The minimal such $p_0$ we call {\em the critical exponent}.
For $X$ a Banach space, $X\otimes(S)_{-p}$ is a Banach space for every $p\geq p_0$, and $X\otimes (S)_{-1}$ is a Frech\'et space. 
%

The Wick product is introduced to overcome the multiplication problem for random variables in \cite{HOUZ1996} and it is generalized to the set of generalized stochastic processes in \cite{LPSZ2015}. Recall, if $F,G \in X \otimes (S)_{-p}$, $p\geq 0$, are generalized stochastic processes given in chaos expansion forms
$F= \sum_{\alpha\in \mathcal{I} }f_\alpha  H_\gamma$ and $G= \sum_{\beta\in \mathcal{I}} g_\beta  H_\beta$, then the Wick product $F \lozenge G,$ is  defined by
$F \lozenge G =\sum_{\gamma \in \mathcal I} \left( \sum_{\alpha + \beta = \gamma} f_\alpha g_\beta \right)  H_\gamma$.

We conclude the introductory section stating a theorem on the deterministic parabolic evolution problems, and a technical Lemma, both necessary for the later analysis. The proof of the theorem is similar to the proof of  Theorem 3 in \cite{GLO2021}, while  proof of the lemma is straightforward, thus both are omitted.

\begin{thm}\label{theorem estimates L infinity deterministic}
	Let the unbounded and closed operator $\mathcal L$ with a dense domain $D\subseteq L^2(\mathbb R^d)$,  be an infinitesimal generator of a $C_0$-semigroup $(T_t)_{t\geq 0}$ on $L^2(\R^d)$. Let the force term $f\in AC([0,T];L^2(\R^d))$, i.e., being differentiable a.e. on $[0,T]$ with $f' \in L^1(0,T;L^2(\R^d))$. Let  the initial condition $g\in D$,  and the potential $q \in L^{\infty} (\mathbb{R}^d)$. Then, the deterministic parabolic  initial value problem 
	\[
	\left( \frac{\partial}{\partial t}  -  {\mathcal L} \right)  u(t, x) + {q (x)} \cdot  u(t, x)  = f(t, x ), 
	\quad u(0, x) = g(x),
	\]
	has a unique bounded nonnegative solution $u\in AC([0,T];L^2(\R^d))$
	which satisfies  
	\begin{equation*} 
	\label{ee-perturbations}
	\|u(t,\cdot)\|_{L^2} \leq  M(t)  \left( \|g(\cdot)\|_{L^{2}} +\int_0^t \|f(s,\cdot)\|_{L^2} \,ds\right), \quad t\in (0,T],
	\end{equation*}
	where 
	$M(t) : = M \exp {\left( \left(w + M \|q\|_{L^\infty}\right)t \right)},$ and $w \in \mathbb{R}$ and $M>0$ are the stability constants from the semigroup estimate $\|T_t\|_{L(L^2(\R^d))}\leq Me^{wt}$,  $t\geq 0.$
\end{thm}
\begin{lem} \label{Lema M(t)} Let $M(t)$ be as in Theorem \ref{theorem estimates L infinity deterministic} with the potential $q_{\mathbf{0}} \in L^{\infty} (\mathbb{R}^d)$. Then, the following estimates hold
	\begin{itemize}
		\item[(a)] 	$\tilde{M}(t) := \displaystyle \int_0^t M(s) \,ds  = \dfrac{M(t)-M}{w+M\|q_{\mathbf{0}} \|_{L^\infty}},$ \quad $\displaystyle \int_0^t s M(s) \,ds \leq t  \tilde{M}(t),$ 
		\item[(b)] $\displaystyle \int_0^t  M(s)  \tilde{M}(s)^n  \,ds \leq	 \  \tilde{M}(t)^{n+1},$ \quad  $\displaystyle \int_0^t s M(s) \tilde{M}(s)^n\,ds \leq  t \tilde{M}(t)^{n+1}$ for all $n \in \N.$
	\end{itemize}
\end{lem}

\section{Stochastic parabolic equations with random and space depending bounded potential} 
\label{sec3}
Now we turn our attention to stochastic initial value problem  \eqref{general problem}, more precisely we consider
\begin{align}\label{Eq:ish2}
\left( \frac{\partial}{\partial t}  -  {\mathcal L} \right)  U(t, x,\omega) + {Q (x,\omega)} \lozenge  U(t, x,\omega)&  = F(t, x ,\omega), \quad t\in (0,T],\, x\in \mathbb R^d,\, \omega \in \Omega\\
U(0, x,\omega) &= G(x,\omega)\quad x\in \mathbb R^d,\, \omega \in \Omega .\nonumber
\end{align} 

The main result  follows.
\begin{thm} \label{Thm0: stochastic weak sol 1} 
	Let $\mathcal L$ be an unbounded closed operator with dense domain $D\subseteq L^2(\mathbb R^d)$, acting on the space component and  generating  a $C_0$-semigroup on $L^2(\R^d)$ satisfying $\|T_t\|_{L(L^2(\R^d))}\leq Me^{wt}$, and let $\tilde{M}(T)$ be as in Lemma \ref{Lema M(t)}. 
	Let the potential $Q \in L^\infty(\mathbb R^d) \otimes (S)_{-1}$ be a generalized stochastic process of Kondratiev-type with the critical exponent $p_1$ and the chaos expansion
	$Q(x, \omega) = \sum\limits_{\gamma\in \mathcal I} q_\gamma( x) \, H_\gamma(\omega)$, with $q_\mathbf{0}\in L^\infty(\mathbb R^d)$ such that $\tilde{M}(T) \|q_\mathbf{0}\|_{L^\infty} \not= 1$, and with $q_\gamma \in L^\infty(\mathbb R^d)$ such that  $\|q_{\gamma}\|_{L^\infty} \leq \|q_\mathbf{0}\|_{L^\infty}$,\\ for all $\gamma \in \mathcal{I}$.
	Further, assume that  the force term $F\in AC([0,T];L^2(\R^d)) \otimes (S)_{-1}$ and  the initial condition $G\in D \otimes (S)_{-1}$ are  generalized stochastic processes of Kondratiev-type, with critical exponents $p_2$ and $p_3$, respectively. 
	Then,  there exists a unique  generalized stochastic process  \linebreak $U \in AC([0,T];D)  \otimes (S)_{-1} \subseteq AC([0,T];L^2 (\R^d)) \otimes (S)_{-1} $ satisfying the stochastic evolution initial value problem \eqref{Eq:ish2}. 
	Moreover,  for all $t\in [0,T]$ the coefficients $u_\gamma$, $\gamma \in \mathcal{I}$ satisfy
	\begin{equation}
	\label{ocena koeficijenti}
	\norm{u_\gamma (t,\cdot)}{L^2}  \leq  M(t)\Bigg\{ a_{\gamma}(t)   + \sum\limits_{k=1}^{|\gamma|} \tilde{M}(t)^k \bigg( \sum_{\substack{0 \leq |\beta| \leq |\gamma|-k\\ \beta<\gamma}} a_\beta (t) \, \bigg(\sum_{\substack{\theta_1+\dots+\theta_k=\gamma-\beta\\\theta_i \not=\mathbf{0}, \, i=1,\dots, k}} \prod_{i=1}^{k} \|q_{\theta_i}\|_{L^\infty} \bigg)\bigg) \Bigg\},
	\end{equation}
	where $a_\gamma (t) :=  \|g_{\gamma}\|_{L^2} + t \| f_{\gamma}\|$, $\gamma\in\mathcal I$.
\end{thm}
\begin{proof}
Representing stochastic processes $Q$, $F$ and $G$ appearing in the problem \eqref{Eq:ish2} in their chaos expansion forms, assuming the solution $U$ in the form
$U(t, x, \omega) = \sum\limits_{\gamma\in \mathcal I} u_\gamma(t, x) H_\gamma(\omega)$, 
and using the definition of the Wick product, we  formally obtain
\[\begin{split}
\sum_{\gamma\in \mathcal I} \left(\frac{\partial}{\partial t}  - \mathcal L \right) \, u_\gamma(t, x) H_\gamma(\omega) + \sum_{\gamma\in \mathcal I} \sum_{\alpha+ \beta=\gamma} q_\alpha(x) \, u_\beta(t, x) \, H_\gamma(\omega) &= \sum_{\gamma\in \mathcal I} f_\gamma(t, x)  \, H_\gamma(\omega),\\
\sum_{\gamma\in \mathcal I} u_\gamma(0,x) H_\gamma(\omega) &= \sum_{\gamma\in \mathcal I} g_\gamma(x) H_\gamma(\omega).
\end{split}\]
From the uniqueness of the chaos expansion representations, the problem is reduced to a triangular system of deterministic equations which can be solved recursively with respect to the length of $\gamma\in \mathcal I$. 
For $|\gamma|=0$: 
\begin{equation}
\label{deterministic system q(x): |gamma|=0}
\left(\frac{\partial}{\partial t}  - \mathcal L  \right) \, u_{\mathbf{0}}(t, x) + \, q_{\mathbf{0}} (x) \,  u_{\mathbf{0}} (t, x) =  f_{\mathbf{0}} (t, x), \quad \quad
u_{\mathbf{0}} (0, x) = g_{\mathbf{0}} (x). 
\end{equation}
By assumptions, 
$q_{\mathbf{0}} \in L^{\infty}(\R^d)$, $f_{\mathbf{0}}\in AC([0,T]; L^2(\R^d))$, and $g_{\mathbf{0}}\in D$. Theorem \ref{theorem estimates L infinity deterministic} implies a unique solution $u_{\mathbf{0}} \in AC([0,T],D)$ to  \eqref{deterministic system q(x): |gamma|=0} given by
$u_0(t,x) = S_t g_{\mathbf{0}}(x) + \int_0^t S_{t-s} f_{\mathbf{0}} (s,x) ds$, $t\in [0,T]$, $x\in \R^d$, where $(S_t)_{t \geq 0}$ is a $C_0$-semigroup on $L^2(\mathbb R^d)$ generated by $\L - q_{\mathbf{0}}\text{Id}$, 
satisfying 
\begin{equation} \label{solution u_0 L2 estimate}
\|u_{\mathbf{0}}(t,\cdot)\|_{L^2} \leq M(t) \left( \|g_{\mathbf{0}} \|_{L^2 } + \int_0^t \| f_{\mathbf{0}} (s,\cdot)\|_{L^2} \,ds \right) \leq M(t) \left( \|g_{\mathbf{0}} \|_{L^2 } + t \| f_{\mathbf{0}}\|  \right)=M(t)a_{\mathbf{0}}(t).
\end{equation}

For $|\gamma|>1$ we have
$$	\left(\frac{\partial}{\partial t}  - \mathcal L  \right) \, u_{\gamma}(t, x) + \, q_{\mathbf{0}} (x) \,  u_{\gamma} (t, x) \, = 	\tilde{f}_{\gamma} (t, x),    \quad
u_{\gamma} (0, x) = g_{\gamma} (x)  ,$$
where $\tilde{f}_{\gamma} (t, x)  = f_{\gamma} (t, x)-\sum_{\substack{\alpha+\beta=\gamma\\ \alpha \not= \mathbf{0}}} q_\alpha(x) \, u_\beta (t, x)=f_{\gamma} (t, x)-\sum_{\mathbf{0}\leq \beta <\gamma} q_{\beta-\gamma}(x) \, u_\beta (t, x)$, with $u_{\mathbf{0}}$ and $ u_\beta$, $\beta<\gamma$ being the solutions  obtained in the previous steps. 
Thus, we obtain a deterministic  problem of the same form as  for $|\gamma|=0$ satisfying, by assumptions, the conditions of  Theorem \ref{theorem estimates L infinity deterministic}  yielding a unique solution $u_{\gamma} \in AC([0,T],D)$ given by  $ u_{\gamma}(t,x) =  S_t g_{\gamma} (x) + \int_0^t S_{t-s} \tilde{f}_{\gamma} (t, x) ds$ and satisfying 
\begin{equation}\label{estFromThm}
\norm{u_\gamma (t,\cdot)}{L^2}  \leq  
M(t)\left( a_\gamma (t) +  \sum\limits_{\mathbf{0} \leq \beta < \gamma} \linf{q_{\gamma-\beta}} \int_0^t \ltwo{ u_{\beta} (s, \cdot)}\,ds\right).
\end{equation}

Next, by induction,  we will prove that $u_\gamma,$ $\gamma \in \mathcal{I}$, satisfy the estimate \eqref{ocena koeficijenti}. The estimate for $|\gamma|=0$ boils down to \eqref{solution u_0 L2 estimate}. We assume that the estimate \eqref{ocena koeficijenti} holds for every $\beta \in \mathcal{I}$ with $|\beta|\leq n$, i.e.,
\begin{equation} \label{opsta formula}
\norm{u_\beta (t,\cdot)}{L^2}  \leq  M(t)\Bigg\{ a_{\beta}(t)   + \sum\limits_{l=1}^{|\beta|} \tilde{M}(t)^l \bigg( \sum_{\substack{0 \leq |\alpha| \leq |\beta|-l\\ \alpha<\beta}} a_\alpha (t) \sum_{\substack{\theta_1+\dots+\theta_l=\beta-\alpha\\\theta_i \not=\mathbf{0}, \, i=1,\dots, l}} \prod_{i=1}^{l} \|q_{\theta_i}\|_{L^\infty} \bigg) \Bigg\},
\end{equation}
and want to show that \eqref{ocena koeficijenti} holds for $\gamma \in \mathcal{I}$ with $|\gamma|=n+1.$ 
Integrating  \eqref{opsta formula} and using Lemma \ref{Lema M(t)} we obtain
$$
\int_0^t \ltwo{ u_{\beta} (s, \cdot)}\,ds \leq a_\beta (t) \tilde{M}(t) + \sum\limits_{l=1}^{|\beta|} \tilde{M}(t)^{l+1} \Bigg( \sum_{\substack{0 \leq |\alpha| \leq |\beta|-l \\ \alpha<\beta}} a_\alpha (t) \sum_{\substack{\theta_1+\dots+\theta_l=\beta-\alpha \\ \theta_i \not=\mathbf{0}}} \prod_{i=1}^{l} \|q_{\theta_i}\|_{L^\infty} \Bigg).
$$ 
Starting from  \eqref{estFromThm} we obtain that $\norm{u_\gamma (t,\cdot)}{L^2}$ is bounded from above by $M(t)$ multiplied with
\begin{equation*}
a_\gamma (t)    
+  \sum\limits_{\mathbf{0} \leq \beta < \gamma} \linf{q_{\gamma-\beta}} a_\beta (t) \tilde{M}(t) 
+  \sum\limits_{\mathbf{0} < \beta < \gamma} \linf{q_{\gamma-\beta}} \sum\limits_{l=1}^{|\beta|} \tilde{M}(t)^{l+1} \bigg( \sum_{\substack{0 \leq |\alpha| \leq |\beta|-l \\ \alpha<\beta}} a_\alpha (t) \sum_{\substack{\theta_1+\dots+\theta_l=\beta-\alpha \\ \theta_i \not=\mathbf{0}}} \prod_{i=1}^{l} \|q_{\theta_i}\|_{L^\infty} \bigg).
\end{equation*}
We want to sum terms with respect to powers of $\tilde{M}(t)$. The first step is to change the order of the first two sums in the third term, where, while $\mathbf{0} < \beta<\gamma$, the length  $|\beta|$  varies from $1$ to $n$, thus 
the third term becomes 
\begin{eqnarray*}
	\sum_{l=1}^n \tilde{M}(t)^{l+1} \sum\limits_{\mathbf{0} < \beta < \gamma} \Bigg( \sum_{\substack{0 \leq |\alpha| \leq |\beta|-l \\ \alpha<\beta}} a_\alpha (t) \sum_{\substack{\theta_1+\dots+\theta_l=\beta-\alpha \\ \theta_i \not=\mathbf{0}}} \prod_{i=1}^{l} \|q_{\theta_i}\|_{L^\infty} \|q_{\gamma-\beta}\|_{L^\infty}\Bigg).
\end{eqnarray*}
Next we merge the sums over $\mathbf{0} < \beta < \gamma$ (which implies $|\beta| \leq |\gamma|-1$) and $\mathbf{0} \leq  \alpha<\beta$ with $ |\alpha| \leq |\beta|-l $ into the sum over $\mathbf{0} \leq \alpha <  \gamma$ with $ |\alpha| \leq |\gamma|-(l+1)$. Denoting $q_{\theta_{l+1}}:= q_{\gamma-\beta}$ we obtain
\begin{eqnarray*}
	&&\sum_{l=1}^n \tilde{M}(t)^{l+1} \sum_{\substack{\mathbf{0} \leq \alpha <\gamma \\ 0 \leq |\alpha| \leq |\gamma|-(l+1) }} a_\alpha (t) \sum_{\substack{\theta_1+\dots+\theta_{l+1}=\gamma-\alpha \\ \theta_i \not=\mathbf{0}}} \prod_{i=1}^{l+1} \|q_{\theta_i}\|_{L^\infty} \\
	&&= \sum_{k=2}^{|\gamma|} \tilde{M}(t)^{k} \sum_{\substack{  0 \leq |\alpha| \leq |\gamma|-k \\\alpha <\gamma }} a_\alpha (t) \sum_{\substack{\theta_1+\dots+\theta_{k}=\gamma-\alpha \\ \theta_i \not=\mathbf{0}}} \prod_{i=1}^{k} \|q_{\theta_i}\|_{L^\infty}.
\end{eqnarray*}
Therefore, $\norm{u_\gamma (t,\cdot)}{L^2} $ is bounded by
\begin{equation*}
M(t)\Bigg\{ a_\gamma (t)    + \sum\limits_{\mathbf{0} \leq \beta < \gamma} \linf{q_{\gamma-\beta}} a_\beta (t) \tilde{M}(t)+ \sum_{k=2}^{|\gamma|} \tilde{M}(t)^{k} \sum_{\substack{  0 \leq |\alpha| \leq |\gamma|-k \\\alpha <\gamma }} a_\alpha (t) \sum_{\substack{\theta_1+\dots+\theta_{k}=\gamma-\alpha \\ \theta_i \not=\mathbf{0}}} \prod_{i=1}^{k} \|q_{\theta_i}\|_{L^\infty} \Bigg\},
\end{equation*}
and thus \eqref{ocena koeficijenti} holds.
%
Next we show that the solution $U$ is a generalized stochastic process of Kondratiev-type, i.e., that the sum $\vertiii{U}^2  := \sum_{ \gamma\in \mathcal{I}}\|u_\gamma\|^2_{AC([0,T]; L^2(\R^d))} (2\N)^{-p\gamma}$  is finite for some critical exponent $p$ to be determined bellow. 
Using  \eqref{ocena koeficijenti} and $(a_1+a_2+a_3)^2 \leq 3a_1^2 + 3a_2^2+3a_3^2$ we find that $\vertiii{U}^2 $ is bounded by
\begin{eqnarray*}
	&&3 M(T)^2\sum_{\gamma \in \mathcal{I}}\| g_{\gamma}\|_{L^2}^2 (2 \N)^{-p \gamma}  + 3T^2M(T)^2 \sum_{\gamma \in \mathcal{I}} \| f_{\gamma}\|^2 (2 \N)^{-p \gamma}\\
	& & +  \, 3 M(T)^2 \sum_{\gamma \in \mathcal{I}}\Bigg(\sum\limits_{k=1}^{|\gamma|} \tilde{M}(T)^k \bigg( \sum_{\substack{0 \leq |\beta| \leq |\gamma|-k \\ \beta<\gamma}} a_\beta (T) \bigg( \sum_{\substack{\theta_1+\dots+\theta_k=\gamma-\beta \\ \theta_i \not=\mathbf{0}}} \prod_{i=1}^{k} \|q_{\theta_i}\|_{L^\infty} \bigg)\bigg) \Bigg)^2(2 \N)^{-p \gamma} \\
	&& = : 3 M(T)^2 \left( S_1+ T^2 S_2 +S_3\right).
\end{eqnarray*}	
Chosing $p \geq \max\{p_2,p_3\}$, according to assumptions on $G$ and $F$, we have
\begin{equation} \label{suma S1+T2 S2}
S_1+T^2 S_2  
\leq  \sum_{\gamma \in \mathcal{I}}\| g_{\gamma}\|_{L^2}^2 (2 \N)^{-p_3 \gamma} + T^2 \sum_{\gamma \in \mathcal{I}}\| f_{\gamma}\|^2 (2 \N)^{-p_2 \gamma} 
:=A < \infty.
\end{equation}
For $S_3$, using  $(\sum\limits_{k=1}^{|\gamma|} x_k)^2 \leq |\gamma| \sum\limits_{k=1}^{|\gamma|} x_k^2$ we obtain
\[
S_3 \leq  \sum_{\gamma \in \mathcal{I}}|\gamma|  \sum\limits_{k=1}^{|\gamma|}  \tilde{M}(T)^{2k} \bigg( \sum_{\substack{0 \leq |\beta| \leq |\gamma|-k \\ \beta<\gamma}} a_\beta (T) \bigg( \sum_{\substack{\theta_1+\dots+\theta_k=\gamma-\beta \\ \theta_i \not=\mathbf{0}}} \prod_{i=1}^{k} \|q_{\theta_i}\|_{L^\infty} \bigg)\bigg)^2 (2\N)^{-p\gamma}.
\]
Further, using $\left( \sum\limits_{ \alpha \in \mathcal{I}} |x_{\alpha} y_\alpha| \right)^2 \leq \left( \sum\limits_{\alpha \in \mathcal{I}}|x_{\alpha}|^2\right) \left(  \sum\limits_{ \alpha \in \mathcal{I}} |y_\alpha|^2 \right)$ and rearranging the powers of $(2\N)^{-p\gamma}$ we obtain
\begin{eqnarray*}
	S_3 \leq \sum_{\gamma \in \mathcal{I}} |\gamma|  \sum\limits_{k=1}^{|\gamma|}  \tilde{M}(T)^{2k} \sum_{\substack{0 \leq |\beta| \leq |\gamma|-k \\ \beta<\gamma}} (a_\beta (T) (2\N)^{-\frac{p\gamma}{6}})^2 
	\sum_{\substack{0 \leq |\beta| \leq |\gamma|-k \\ \beta<\gamma}} 
	\bigg( \sum_{\substack{\theta_1+\dots+\theta_k=\gamma-\beta \\ \theta_i \not=\mathbf{0}}} \prod_{i=1}^{k} \|q_{\theta_i}\|_{L^\infty} (2\N)^{-\frac{p\gamma}{6}}\bigg)^2  (2\N)^{-\frac{p\gamma}3}.
\end{eqnarray*}
Since $\beta<\gamma$ and $p \geq \max\{p_2,p_3\}$ (chosen in \eqref{suma S1+T2 S2}) we have
\begin{equation*} 
\sum_{\substack{0 \leq |\beta| \leq |\gamma|-k \\ \beta<\gamma}} a_\beta (T)^2  (2\N)^{-\frac{p\gamma}3} \leq \sum_{\substack{0 \leq |\beta| \leq |\gamma|-k \\ \beta<\gamma}} a_\beta (T)^2  (2\N)^{-\frac{p\beta}3} \leq 2 A <\infty,
\end{equation*}
and by Lemma \ref{pomocna lema - ocene} we find
\[ S_3	 \leq  2A \sum_{\gamma \in \mathcal{I}} |\gamma|  \sum\limits_{k=1}^{|\gamma|}  \tilde{M}(T)^{2k}  \Bigg( \sum_{\substack{0 \leq |\beta| \leq |\gamma|-k \\ \beta<\gamma}} \|q_\textbf{0}\|^{2k}_{L^\infty}  \,  2^{2k|\gamma-\beta|}(2\N)^{-\frac{p(\gamma-\beta)}{3}}\Bigg) (2\N)^{-\frac{p\gamma}3}.
\]
By Lemma \ref{pomocna lema - ocene}, there exists $s\geq 0$ such that $2^{2k|\gamma-\beta|}\leq (2\N)^{\frac{s (\gamma-\beta)}3}$, and choosing $p$ to satisfy $p>s+3$ we have 
\begin{eqnarray*}
	S_3	
	&\leq & 2A \sum_{\gamma \in \mathcal{I}}  |\gamma|   \sum\limits_{k=1}^{|\gamma|}  \tilde{M}(T)^{2k}  \|q_\textbf{0}\|^{2k}_{L^\infty} \Bigg( \sum_{\substack{ |\gamma-\beta| \geq k \\ \beta<\gamma}} (2\N)^{-\frac{(p-s) (\gamma-\beta)}3}\Bigg) (2\N)^{-\frac{p\gamma}{3}}\\
	& \leq & 2 A C \sum_{\gamma \in \mathcal{I}}  |\gamma|   \sum\limits_{k=1}^{|\gamma|}  \tilde{M}(T)^{2k} \|q_\textbf{0}\|^{2k}_{L^\infty}   (2\N)^{-\frac{p\gamma}3},
\end{eqnarray*}
where $A$ is defined in \eqref{suma S1+T2 S2} and $C:=\sum_{\alpha \in \mathcal I} (2\N)^{-\frac{(p-s) \alpha}3}$ is finite by \eqref{2N konv p>1}. The assumption \mbox{$\tilde{M}(T) \|q_\textbf{0}\|_{L^\infty} \not= 1$} allows to
sum up the inner sum leading to
\begin{eqnarray*}	
	S_3 \leq 2 A C \dfrac{\tilde{M}(T)^{2}\|q_\textbf{0}\|^{2}_{L^\infty}}{1-\tilde{M}(T)^{2}\|q_\textbf{0}\|^{2}_{L^\infty}} \left(  \sum_{\gamma \in \mathcal{I}}  |\gamma| (2\N)^{-\frac{p\gamma}3}  - \sum_{\gamma \in \mathcal{I}}  |\gamma| \tilde{M}(T)^{2|\gamma|}\|q_\textbf{0}\|^{2|\gamma|}_{L^\infty} (2\N)^{-\frac{p\gamma}3} \right).
\end{eqnarray*}

By Lemma \ref{pomocna lema - ocene}, 
for $\tilde{M} (T)^2 \|q_\textbf{0}\|^2 >0$ there exists $s_1>0$ so that
\begin{eqnarray} \label{s3 suma 2}
\sum_{\gamma \in \mathcal{I}}  |\gamma| \tilde{M}(T)^{2|\gamma|}\|q_\textbf{0}\|^{2|\gamma|}_{L^\infty} (2\N)^{-\frac{p\gamma}3} 
\leq \sum_{\gamma \in \mathcal{I}} (2\N)^{s_1\gamma} (2\N)^{-(\frac{p}3-1)\gamma}.
\end{eqnarray}

Choosing $p>3s_1+6$, by Lemma \ref{pomocna lema - ocene} we have that  \eqref{s3 suma 2} is finite, and
\begin{equation*}
\sum_{\gamma \in \mathcal{I}}  |\gamma| (2\N)^{-\frac{p\gamma}3} \leq  \sum_{\gamma \in \mathcal{I}}  (2\N)^\gamma (2\N)^{-\frac{p\gamma}3} =  \sum_{\gamma \in \mathcal{I}}   (2\N)^{-(\frac{p}3-1)\gamma} < \infty,
\end{equation*}
yielding that $S_3$ is finite.
Finally, $\vertiii{U}^2 < \infty$  for $p\geq \max\{p_2, p_3, s+3, 3s_1+6\}$, and  \linebreak $U \in AC([0,T]; D)  \otimes (S)_{-1}$. \\

The uniqueness of the solution $U$ follows from the uniqueness of its coefficients $u_\gamma$, $\gamma\in \mathcal I$ and the uniqueness of  the chaos expansion representation in the Fourier--Hermite basis of orthogonal stochastic polynomials.
\end{proof}

%
%
%

\section*{Acknowledgments}

Author Ljubica Oparnica is supported by the FWO Odysseus 1 grant no. G.0H94.18N: Analysis and Partial Differential Equations. Author Sne\v zana Gordi\'c is supported by Ministry of Education, Science and Technological Development of the Republic of Serbia, by Faculty of Education in Sombor.

\end{document}